\documentclass[11pt,a4paper,oneside,reqno]{amsart}

\usepackage{latexsym}
\usepackage{amsmath}
\usepackage{amsthm}
\usepackage{amssymb}
\usepackage{amsfonts}
\usepackage{fancyhdr}
\usepackage{comment}

\usepackage{mathrsfs}                           

\newcommand{\mR}{\mathbb{R}}                    
\newcommand{\mC}{\mathbb{C}}                    
\newcommand{\abs}[1]{\lvert #1 \rvert}          
\newcommand{\norm}[1]{\lVert #1 \rVert}         
\newcommand{\br}[1]{\langle #1 \rangle}         

\newcommand{\mSp}{\mathscr{S}^{\prime}}

\newcommand{\mF}{\mathscr{F}}

\newcommand{\id}{\mathrm{Id}}

\theoremstyle{definition}
\newtheorem{thm}{Theorem}[section]
\newtheorem{prop}[thm]{Proposition}
\newtheorem{cor}[thm]{Corollary}
\newtheorem{lemma}[thm]{Lemma}

\newtheorem*{definition}{Definition}

\numberwithin{equation}{section}

\title[Reconstructions on admissible manifolds]{Reconstructions from boundary measurements on admissible manifolds}

\author{Carlos E. Kenig}
\address{Department of Mathematics, University of Chicago}
\email{cek@math.uchicago.edu}
\author{Mikko Salo}
\address{Department of Mathematics and Statistics, University of Helsinki}
\email{mikko.salo@helsinki.fi}
\author{Gunther Uhlmann}
\address{Department of Mathematics, University of Washington and Department of Mathematics, University of California, Irvine}
\email{gunther@math.washington.edu}

\date{November 2, 2010}

\begin{document}

\begin{abstract}
We prove that a potential $q$ can be reconstructed from the Dirichlet-to-Neumann map for the Schr\"odinger operator $-\Delta_g + q$ in a fixed admissible $3$-dimensional Riemannian manifold $(M,g)$. We also show that an admissible metric $g$ in a fixed conformal class can be constructed from the Dirichlet-to-Neumann map for $\Delta_g$. This is a constructive version of earlier uniqueness results by Dos Santos Ferreira et al. \cite{DKSaU} on admissible manifolds, and extends the reconstruction procedure of Nachman \cite{N} in Euclidean space. The main points are the derivation of a boundary integral equation characterizing the boundary values of complex geometrical optics solutions, and the development of associated layer potentials adapted to a cylindrical geometry.
\end{abstract}

\maketitle


\section{Introduction} \label{section_introduction}

This paper is concerned with the problem of reconstructing material parameters of a medium from boundary measurements. A typical question of this type is Calder\'on's inverse conductivity problem \cite{C}, which consists in recovering the conductivity of a body from voltage to current measurements at the boundary. For bounded domains in Euclidean space in dimensions $n \geq 3$, it was proved in \cite{SU} that a smooth positive scalar conductivity $\sigma$ is uniquely determined by the Dirichlet-to-Neumann map (DN map) $\Lambda_{\sigma}$ representing the boundary measurements. This uniqueness result was then extended to a reconstruction procedure in \cite{N} and independently in \cite{No_1988}, see also \cite{HN}. In two dimensions, uniqueness and reconstruction for this problem was proved even for bounded measurable conductivities in \cite{AP_boundary}, \cite{AP}.

In this paper we consider Calder\'on's inverse problem and related questions in anisotropic media, where the conductivity depends on direction. This corresponds to replacing the scalar conductivity $\sigma$ by a smooth symmetric positive definite matrix. The question then is to recover the matrix $\sigma$ from the DN map $\Lambda_{\sigma}$, up to the natural obstruction given by diffeomorphisms which fix the boundary. If $n=2$, it is proved in \cite{ALP} that any bounded measurable matrix conductivity $\sigma$ is determined by $\Lambda_{\sigma}$ up to diffeomorphism. For constructive results see \cite{HS} and the references therein.

In three and higher dimensions the anisotropic Calder{\'o}n problem is open even for smooth matrix conductivities. We refer to \cite{DKSaU} for a more thorough discussion and references to known results. It was observed in \cite{LeU}  that the anisotropic Calder\'on problem is closely related to certain inverse problems for the Laplace-Beltrami operator on a Riemannian manifold, which we set out to define.

{\bf Statement of main results.} Let $(M,g)$ be a compact oriented Riemannian manifold with $C^{\infty}$ boundary, and let $\Delta_g$ be the Laplace-Beltrami operator. In local coordinates 
$$
\Delta_g u = \abs{g}^{-1/2} \sum_{j,k=1}^n \frac{\partial}{\partial x_j} \left( \abs{g}^{1/2} g^{jk} \frac{\partial u}{\partial x_k} \right)
$$
where $g = (g_{jk})$ is the metric in local coordinates, $(g^{jk})$ is the inverse matrix of $(g_{jk})$, and $\abs{g} = \det(g_{jk})$. Consider the Dirichlet problem 
\begin{equation*}
\left\{ \begin{array}{rll}
\Delta_g u &\!\!\!= 0 & \quad \text{in } M, \\
u &\!\!\!= f & \quad \text{on } \partial M.
\end{array} \right.
\end{equation*}
For any $f \in H^{3/2}(\partial M)$ there is a unique solution $u \in H^2(M)$, and the DN map is defined by 
$$
\Lambda_g: H^{3/2}(\partial M) \to H^{1/2}(\partial M), \ \ f \mapsto \partial_{\nu} u|_{\partial M}
$$
where the normal derivative is given by 
$$
\partial_{\nu} u|_{\partial M} = \sum_{j,k=1}^n g^{jk} \frac{\partial u}{\partial x_j} \nu_k.
$$
Here $\nu_k = \sum_{k,l=1}^n g_{kl} \nu^l$, and $(\nu^1,\ldots,\nu^n)$ is the coordinate expression for the unit outer normal vector $\nu$ on $\partial M$.

Our first result states that the map $\Lambda_g$ constructively determines $g$ within a known conformal class of admissible metrics (as defined below).

\begin{thm} \label{thm_main1}
Let $(M,g)$ be a given admissible $3$-dimensional manifold. If $c$ is a smooth positive function on $M$, then from the knowledge of $\Lambda_{cg}$ one can constructively determine $c$.
\end{thm}

The next question concerns an inverse problem for the Schr\"odinger equation in $(M,g)$. If $q$ is a smooth function on $M$, we consider the Dirichlet problem 
\begin{equation*}
\left\{ \begin{array}{rll}
(-\Delta_g + q) u &\!\!\!= 0 & \quad \text{in } M, \\
u &\!\!\!= f & \quad \text{on } \partial M.
\end{array} \right.
\end{equation*}
We make the standing assumption that 
$$
\text{$0$ is not a Dirichlet eigenvalue of $-\Delta_g + q$ in $M$.}
$$
This means that for any $f \in H^{3/2}(\partial M)$ the equation has a unique solution $u \in H^2(M)$, and the DN map can be defined by 
$$
\Lambda_{g,q}: H^{3/2}(\partial M) \to H^{1/2}(\partial M), \ \ f \mapsto \partial_{\nu} u|_{\partial M}.
$$
The second main result is as follows.

\begin{thm} \label{thm_main2}
Let $(M,g)$ be a given admissible $3$-dimensional manifold. If $q$ is a smooth function on $M$, then from the knowledge of $\Lambda_{g,q}$ one can constructively determine $q$.
\end{thm}

To complete the statement of the main results, let us give the definition of admissible manifolds. These arose in \cite{DKSaU} as the first class of manifolds beyond real-analytic or Einstein ones for which one can prove uniqueness results for the anisotropic inverse problems described above.

\begin{definition}
A compact oriented Riemannian manifold $(M,g)$ with smooth boundary is \emph{admissible} if $\dim(M) \geq 3$ and if $(M,g) \subset \subset (T^{\text{int}},g)$ where $T = \mR \times M_0$ is a cylinder with metric $g = c(e \oplus g_0)$ (here $c$ is a smooth positive function and $e$ is the Euclidean metric on $\mR$), and $(M_0,g_0)$ is an $(n-1)$-dimensional simple manifold.
\end{definition}

Thus, up to a conformal factor, an admissible manifold is embedded in a cylinder $(\mR \times M_0, e \oplus g_0)$ and therefore has a Euclidean direction. The transversal manifold $(M_0,g_0)$ needs further to be simple:

\begin{definition}
A compact manifold $(M_0,g_0)$ with boundary is \emph{simple} if for any $p \in M_0$ the exponential map $\exp_p$ is a diffeomorphism from its maximal domain in $T_p M_0$ onto $M_0$, and if $\partial M_0$ is strictly convex (meaning that the second fundamental form is positive definite).
\end{definition}

Examples of admissible manifolds include subdomains of the model spaces (Euclidean space, sphere minus a point, hyperbolic space), sufficently small subdomains of any conformally flat manifold, and domains in $\mR^n$ equipped with a metric of the form 
$$
g(x_1,x') = c(x) \left( \begin{array}{cc} 1 & 0 \\ 0 & g_0(x') \end{array} \right)
$$
where $c$ is a positive smooth function and where $g_0$ is a simple metric in the $x'$ variables.

The unique determination results corresponding to Theorems \ref{thm_main1} and \ref{thm_main2} were proved in \cite{DKSaU}, where also earlier work is discussed and further references are given. There are several recent results that are concerned with the two-dimensional case. If $M$ is a domain in $\mR^2$, \cite{Bu} proved the uniqueness result corresponding to Theorem \ref{thm_main2} and briefly discussed a reconstruction procedure. A proof with constructive character was given in \cite{GT_magnetic} for arbitrary Riemann surfaces $(M,g)$ with boundary (also for the magnetic Schr\"odinger operator), based on earlier nonconstructive proofs in \cite{GT}, \cite{IUY}. Various constructive results for the two-dimensional case, also on Riemann surfaces, appear in \cite{HM}, \cite{HN_2010}, \cite{HS}, \cite{HS2}, and for the three-dimensional case an improved reconstruction result is given in \cite{No_constructive}.

{\bf Outline of proof.} As mentioned, the uniqueness results corresponding to Theorems \ref{thm_main1} and \ref{thm_main2} were proved in \cite{DKSaU}, but the proofs were not constructive. The main point of the present paper is to give constructive proofs, following the well-known reconstruction procedure of Nachman \cite{N} in the case where $M$ is a bounded domain in $\mR^n$ and $g$ is the Euclidean metric. We do not make any claims about the practicality of the reconstruction procedure, but we do prove that all the steps in the corresponding uniqueness proofs in \cite{DKSaU} can be carried out in a constructive way.

The argument in \cite{N} involves \emph{complex geometrical optics} (CGO) solutions $u = e^{-\zeta \cdot x}(1+r)$ to the Schr\"odinger equation $(-\Delta+q)u = 0$ in $\mR^n$, and relies in a crucial way on a uniqueness notion for these solutions upon fixing decay at infinity. One has several equivalent ways of characterizing these solutions, and in particular it is possible to recover the boundary value $u|_{\partial M}$ as the unique solution to a boundary integral equation on $\partial M$ involving $\Lambda_{g,q}$ and other known quantities.

CGO solutions on admissible Riemannian manifolds were constructed in \cite{DKSaU} by Carleman estimates. The solutions were given in a compact manifold, and the construction did not involve a notion of  uniqueness. The paper \cite{KSaU} introduced a direct Fourier analytic construction of CGO solutions, valid in the cylinder $T$ and with a uniqueness notion obtained by fixing a decay condition in the Euclidean variable and Dirichlet boundary values on $\partial T$. We shall use the solutions constructed in \cite{KSaU} to prove Theorem \ref{thm_main2} (which implies Theorem \ref{thm_main1} after a simple reduction).

We next sketch the proof of Theorem \ref{thm_main2}. Let $(M,g) \subset \subset (T^{\text{int}},g)$ be an admissible manifold, and assume that $g = e \oplus g_0$ where $(M_0,g_0)$ is simple. Here we suppose, for simplicity, that $c=1$. We assume that $\Lambda_{g,q}$ and $(M,g)$ are known (thus also $\Lambda_{g,0}$ is known), and use the basic integral identity 
\begin{equation} \label{int_id}
\int_{\partial M} (\Lambda_{g,q}-\Lambda_{g,0})(u|_{\partial M}) v \,dS = \int_M q u v \,dV
\end{equation}
which is valid for any solutions $u, v \in H^2(M)$ of $(-\Delta_g + q)u = 0$ and $-\Delta_g v = 0$ in $M$. We take $u$ and $v$ to be suitable CGO solutions such that $u|_{\partial M}$ may be obtained from $\Lambda_{g,q}$ as the unique solution of a boundary integral equation, and $v|_{\partial M}$ is explicitly given. Then the left hand side of \eqref{int_id} is known. Taking the limit as $\tau \to \infty$ and varying certain parameters in the solutions, we recover in this way the quantity 
\begin{equation} \label{att_raytransform}
\int_0^{\infty} e^{-2\lambda r} \left[ \int_{-\infty}^{\infty} e^{2i\lambda x_1} q(x_1,r,\theta) \,dx_1 \right] \,dr
\end{equation}
for any $\theta \in S^{n-2}$, where $\lambda$ is any nonzero real number and $(r,\theta)$ are  polar normal coordinates in $M_0$ with center on $\partial M_0$. We have extended $q$ into $T \smallsetminus M$ as a function in $C_c^{\infty}(T^{\text{int}})$.

Now, since $(r,\theta)$ are polar normal coordinates, the curves $\gamma: r \mapsto (r,\theta)$ are unit speed geodesics in $(M_0,g_0)$ for any fixed $\theta$. Denoting the quantity in brackets in \eqref{att_raytransform} by $f_{\lambda}(r,\theta)$, it follows that we have recovered 
$$
\int_{\gamma} e^{-2\lambda r} f_{\lambda}(\gamma(r)) \,dr
$$
for any maximal geodesic $\gamma$ going from $\partial M_0$ into $M_0$, and for any $\lambda \neq 0$. This is the attenuated geodesic ray transform of $f_{\lambda}$ with constant attenuation $-2\lambda$, see \cite{DKSaU}, \cite{Sh}. For any $\lambda$ such that this transform can be inverted, we recover $f_{\lambda}$ which is just the one-dimensional Fourier transform 
$$
\mF_{x_1} \{q(\,\cdot\,,x')\}(-2 \lambda) \quad \text{for all } x' \in M_0.
$$

It was proved in \cite{DKSaU} and \cite{FSU} that the attenuated ray transform is invertible for small $\abs{\lambda}$, thus giving information on $\mF_{x_1} \{q(\,\cdot\,,x')\}$ for small frequencies. This determines the compactly supported function $q(\,\cdot\,,x')$ uniquely by the Paley-Wiener theorem. To make this step more constructive one would like to invert the attenuated ray transform for all $\lambda$, which would yield $q(\,\cdot\,,x')$ by taking the inverse Fourier transform. Up to now the argument has been valid for $\dim(M) \geq 3$. However, if $\dim(M)=3$ then $(M_0,g_0)$ is a 2D simple manifold, and the recent result \cite{SaU} shows that the attenuated ray transform is invertible for any attenuation. We can use the inversion procedure in \cite{SaU} to conclude the reconstruction of $q$ from $\Lambda_{g,q}$ if $(M,g)$ is $3$-dimensional.

{\bf Boundary integral equation.} The main new point in the proof is a Fredholm boundary integral equation characterizing the values on $\partial M$ of suitable CGO solutions in $T$. This equation has the form 
\begin{equation} \label{intro_boundaryintegralequation}
(\id + \gamma S_{\tau}(\Lambda_{g,q}-\Lambda_{g,0}))f = u_0 \quad \text{on } \partial M.
\end{equation}
Here $u_0$ is an explicit function on $\partial M$ depending on various parameters, $\gamma$ is the trace operator $H^2(M) \to H^{3/2}(\partial M)$, and $S_{\tau}$ is a special single layer potential depending on a large parameter $\tau$ and adapted to the CGO solutions and the geometry of the cylinder $(T,g) = (\mR \times M_0, e \oplus g_0)$. In fact, we have 
$$
S_{\tau} f(x) = \int_{\partial M} K_{\tau}(x,y) f(y) \,dS(y), \quad x \in T^{\text{int}} \smallsetminus \partial M,
$$
where the integral kernel $K_{\tau}(x,y)$ is explicitly determined by $\tau$ and by the Dirichlet eigenvalues and eigenfunctions of the Laplace-Beltrami operator on $(M_0,g_0)$. We establish basic properties of the single layer operator in Section \ref{section_boundarylayerpotentials}.

We prove that for suitable choices of $u_0$ and for $\abs{\tau}$ large, the equation \eqref{intro_boundaryintegralequation} has a unique solution $f \in H^{3/2}(\partial M)$, and one has $f = u|_{\partial M}$ where $u$ is the corresponding CGO solution. Since the operator on the left hand side of \eqref{intro_boundaryintegralequation} is determined by the boundary measurements and since $u_0$ is explicit, we can indeed determine the boundary values of CGO solutions by solving this Fredholm integral equation.

This approach is analogous to \cite{N} which considers the Euclidean case, except that the uniqueness notion for CGO solutions is obtained from decay conditions and Dirichlet boundary values on the cylinder $(T,g)$ instead of a decay condition at infinity in $\mR^n$. Recently in \cite{NS} another constructive approach appeared. There the boundary integral equation is obtained via Carleman estimates in $M$, and no extension of $M$ to a larger set is needed. It is presumable that a similar approach would work in our case. However, the single layer potential obtained from Carleman estimates is perhaps not so explicit as the operator $S_{\tau}$ introduced above.

We remark here that it would be interesting to establish reconstruction results corresponding to Theorems \ref{thm_main1} and \ref{thm_main2} for the magnetic Schr\"odinger equation or for the time-harmonic Maxwell equations. Uniqueness results for these equations are proved in \cite{DKSaU}, \cite{KSaU}, and constructive results in the Euclidean case appear in \cite{OPS}, \cite{Sa}.

{\bf Structure of paper.} Section \ref{section_introduction} is the introduction. The basic properties of the single layer operator $S_{\tau}$ and related Faddeev Green functions are considered in Section \ref{section_boundarylayerpotentials}. In Section \ref{section_boundaryintegralequation} we introduce several equivalent ways of characterizing CGO solutions, including the required boundary integral equation. The results in Sections \ref{section_boundarylayerpotentials} and \ref{section_boundaryintegralequation} are in fact valid for any transversal manifold $(M_0,g_0)$ with smooth boundary (not necessarily simple). The proofs of Theorems \ref{thm_main1} and \ref{thm_main2} are given in Section \ref{section_proofs}.

\subsection*{Acknowledgements}
C.E.K.~is partly supported by the NSF grant DMS-0968472. M.S.~is supported in part by the Academy of Finland. G.U.~is partly supported by NSF, a Chancellor Professorship at UC Berkeley and a Senior Clay Award.


\section{Boundary layer potentials} \label{section_boundarylayerpotentials}

{\bf Notation and function spaces.} In this section we assume that $(M,g)$ is a compact manifold with smooth boundary, having dimension $n \geq 3$, and that $(M,g) \subset \subset (T^{\text{int}},g)$ where $T = \mR \times M_0$ and $g = e \oplus g_0$, and $(M_0,g_0)$ is any compact $(n-1)$-dimensional manifold with boundary (no restrictions on the metric $g_0$). Points of $T$ are written as $x = (x_1,x')$ where $x_1$ is the Euclidean variable and $x'$ is a point in the transversal manifold $M_0$.

We write $\langle \,\cdot\,,\,\cdot\, \rangle$ for the inner product of tangent vectors, $1$-forms, and other tensors, and $\abs{\,\cdot\,}$ for the norm. The volume element in $(T,g)$ is 
$$
dV(x) = dV_g(x) = dx_1 \,dV_{g_0}(x')
$$
with $dV_{g_0}$ the volume element in $(M_0,g_0)$. We also write $\Gamma = \partial M \subset \subset T$, and denote by $dS$ the volume element on $\Gamma$.

Let $L^2(T) = L^2(T,dV)$ be the standard $L^2$ space in $T$, and let $H^s(T)$ be the corresponding $L^2$ Sobolev spaces. Since $M_0$ is compact, we define 
\begin{gather*}
H^s_{\text{loc}}(T) = \{ f \,;\, f \in H^s([-R,R] \times M_0) \text{ for any $R > 0$} \}, \\
H^s_c(T) = \{ f \in H^s(T) \,;\, f(x_1,x') = 0 \text{ when $\abs{x_1} \geq R$ for some $R > 0$} \}.
\end{gather*}
Writing $\br{t} = (1+t^2)^{1/2}$, we introduce for $\delta \neq 0$, $s \geq 0$ the weighted spaces 
\begin{gather*}
L^2_{\delta}(T) = \{ f \in L^2_{\text{loc}}(T) \,;\, \br{x_1}^{\delta} f \in L^2(T) \}, \\
H^s_{\delta}(T) = \{ f \in H^s_{\text{loc}}(T) \,;\, \br{x_1}^{\delta} f \in H^s(T) \}, \\
H^1_{\delta,0}(T) = \{ f \in H^1_{\delta}(T) \,;\, f|_{\partial T} = 0 \}, \\ H^1_{\text{loc},0}(T) = \{ f \in H^1_{\text{loc}}(T) \,;\, f|_{\partial T} = 0 \},
\end{gather*}
and 
$$
H^1_{-\infty,0}(T) = \bigcup_{\delta \in \mR} H^1_{\delta,0}(T).
$$
Also, $H^1_0(T) = \{ f \in H^1(T) \,;\, f|_{\partial T} = 0 \}$. We define, in the $L^2(T)$ duality, 
$$
H^{-1}(T) = (H^1_0(T))^*.
$$
On $\Gamma = \partial M$ we consider the usual space $L^2(\Gamma) = L^2(\Gamma,dS)$ and the corresponding Sobolev spaces $H^s(\Gamma)$.

Let $0 < \lambda_1 \leq \lambda_2 \leq \ldots$ be the Dirichlet eigenvalues of $-\Delta_{g_0}$ in $(M_0,g_0)$, and let $\{\phi_l\}_{l=1}^{\infty}$ be an orthonormal basis of $L^2(M_0)$ consisting of Dirichlet eigenfunctions, 
$$
-\Delta_{g_0} \phi_l = \lambda_l \phi_l \ \ \text{in } M_0, \quad \phi_l \in H^1_0(M_0).
$$
We write $\text{Spec}(-\Delta_{g_0}) = \{\lambda_l\}_{l=1}^{\infty}$. If $f \in L^2(T)$ we consider the partial Fourier coefficients 
$$
\tilde{f}(x_1,l) = (f(x_1,\,\cdot\,), \phi_l)_{L^2(M_0)}.
$$
The Parseval identity implies that 
$$
\norm{f}_{L^2(T)}^2 = \int_{-\infty}^{\infty} \sum_{l=1}^{\infty} \abs{\tilde{f}(x_1,l)}^2 \,dx_1.
$$

{\bf Standard single layer operator.} Let $K_0$ be the usual inverse of the Dirichlet Laplacian on $T$, defined as follows: since any $u \in H^1_0(T)$ satisfies the Poincar\'e inequality 
\begin{align*}
\int_T \abs{u}^2 \,dV &= \int_{-\infty}^{\infty} \int_{M_0} \abs{u(x_1,x')}^2 \,dV_{g_0} \,dx_1 \\
 &\leq \int_{-\infty}^{\infty} C \int_{M_0} \abs{d_{x'} u(x_1,x')}^2 \,dV_{g_0} \,dx_1 \leq C \int_T \abs{du}^2 \,dV,
\end{align*}
the bilinear form $B(u,v) = \int_T \langle du, dv \rangle \,dV$ is bounded and coercive on $H^1_0(T)$. Consequently, for any $F \in H^{-1}(T)$ there is a unique weak solution $u = K_0 F$ in $H^1_0(T)$ of the equation $-\Delta_g u = F$ in $T$. Thus $K_0$ is a bounded linear operator 
$$
K_0: H^{-1}(T) \to H^1_0(T), \quad -\Delta_g K_0 = \id.
$$
We also consider the trace operator which restricts functions to $\Gamma$, 
$$
\gamma: H^1_0(T) \to H^{1/2}(\Gamma), \quad \gamma u = u|_{\Gamma}.
$$
The adjoint of $\gamma$ satisfies 
\begin{gather*}
\gamma^*: H^{-1/2}(\Gamma) \to H^{-1}(T), \\
(\gamma^* h, \varphi)_{L^2(T)} = (h,\gamma \varphi)_{L^2(\Gamma)} \ \ \text{for any } \varphi \in H^1_0(T).
\end{gather*}
Thus formally $\gamma^* h = h \,dS$.

\begin{definition}
The standard single layer operator on $T$ is the map 
$$
S_0 = K_0 \gamma^*: H^{-1/2}(\Gamma) \to H^1_0(T).
$$
\end{definition}

The next result gives the basic jump and mapping properties of $S_0$. Here we write $M_- = M^{\text{int}}$, $M_+ = T \smallsetminus M$, and $\gamma_{\mp} u = (u|_{M_{\mp}})|_{\Gamma}$ for the restriction of $u$ to $\Gamma$ from the interior or exterior. If $u$ is a function on $T \smallsetminus \Gamma$ such that $u|_{M_{\mp}}$ are $H^1$ in $M_{\mp}$ and satisfy $-\Delta_g u = 0$ in $M_{\mp}$, we define the normal derivatives from the interior or exterior weakly as elements of $H^{-1/2}(\Gamma)$ by 
$$
((\partial_{\nu} u)_{\mp}, h)_{L^2(\Gamma)} = \pm (d(u|_{M_{\mp}}),de_h)_{L^2(M_{\mp})}, \quad h \in H^{1/2}(\Gamma),
$$
where $e_h \in H^1(T)$ is any function with $e_h|_{\Gamma} = h$ and $e_h|_{\partial T} = 0$. The jumps on $\Gamma$ are defined by 
\begin{align*}
[u]_{\Gamma} &= \gamma_- u - \gamma_+ u, \\
[\partial_{\nu} u]_{\Gamma} &= (\partial_{\nu} u)_- - (\partial_{\nu} u)_+.
\end{align*}

\begin{lemma}
If $f \in H^{-1/2}(\Gamma)$ then $u = S_0 f \in H^1_0(T)$ satisfies 
\begin{align*}
-\Delta_g u &= 0 \quad \text{in } M_{\pm}, \\
[u]_{\Gamma} &= 0, \\
[\partial_{\nu} u]_{\Gamma} &= f.
\end{align*}
If $f \in H^s(\Gamma)$ for $s \geq -1/2$ then $u|_{M_-} \in H^{s+3/2}(M_-)$ and $u|_{M_+} = \tilde{u}|_{M_+}$ for some $\tilde{u} \in H^{s+3/2}_{\text{loc}}(T) \cap H^1_0(T)$.
\end{lemma}
\begin{proof}
Harmonicity and the first jump property are direct consequences of the properties of $K_0$. The definitions also imply that for $h \in H^{1/2}(\Gamma)$, 
$$
([\partial_{\nu} u]|_{\Gamma},h)_{L^2(\Gamma)} = (du,de_h)_{L^2(T)} = (\gamma^* f,e_h)_{L^2(T)} = (f,h)_{L^2(\Gamma)}.
$$
This shows the second jump property. If $f \in H^{k+1/2}(\Gamma)$ with $k \geq 0$, then $u$ satisfies 
\begin{align*}
-\Delta_g u &= 0 \quad \text{in } M_{\pm}, \\
[u]_{\Gamma} &= 0, \\
[\partial_{\nu} u]_{\Gamma} &= f \in H^{k+1/2}(\Gamma).
\end{align*}
The transmission property \cite[Theorem 4.20]{McLean} implies that $u|_{M_{\pm}}$ are $H^{k+2}$ near $\Gamma$. Then the properties for $u$ follow from standard interior and boundary regularity for $-\Delta_g$ and from interpolation.
\end{proof}

\begin{cor}
The trace single layer potential satisfies for $s \geq -1/2$ 
$$
\gamma S_0: H^s(\Gamma) \to H^{s+1}(\Gamma).
$$
\end{cor}

Since $K_0$ maps $C^{\infty}_c(T^{\text{int}})$ to $L^2(T^{\text{int}})$, the Schwartz kernel theorem shows that there is a distributional integral kernel $K_0(x,y)$. (We restrict to $T^{\text{int}}$ to avoid having to talk about distributions on manifolds with boundary.) Then formally 
$$
S_0 f(x) = \int_{\Gamma} K_0(x,y) f(y) \,dS(y).
$$
We will need some basic properties of the integral kernel. These are easily obtained by comparing to the Green function on compact manifolds \cite{aubin}, \cite{taylor_toolspde}.

\begin{lemma}
The kernel $K_0$ is smooth in $T^{\text{int}} \times T^{\text{int}}$ away from the diagonal, and it satisfies (with $d$ the Riemannian distance) 
$$
\abs{K_0(x,y)} \leq C_U d(x,y)^{2-n}, \quad x,y \in \overline{U} \subset \subset T^{\text{int}}.
$$
Also, $K_0(x,y) = K_0(y,x)$, and for any $x \in T^{\text{int}}$ one has $\Delta_g(K_0(x,\,\cdot\,)) = 0$ in $T^{\text{int}} \smallsetminus \{x\}$.
\end{lemma}
\begin{proof}
The condition $K_0(x,y) = K_0(y,x)$ follows since $-\Delta_g$ is symmetric. Let $U \subset \subset W \subset \subset T^{\text{int}}$ where $U$ is open and $(\overline{W},g)$ is a compact manifold with smooth boundary, and let $G(x,y)$ be the Dirichlet Green function for the Laplacian in $(\overline{W},g)$ \cite[Section 4.2]{aubin}. We will prove that 
$$
K_0(x,y) = G(x,y) + R(x,y), \quad x,y \in \overline{U}
$$
where $R \in C^{\infty}(\overline{U} \times \overline{U})$. Since the function $G(x,y)$ has the stated properties by \cite[Theorem 4.17]{aubin}, they also follow for $K_0$ by using smoothness of $R$ and simple arguments.

Consider $\varphi \in C^{\infty}_c(U)$ and let $v = u_0 - u_1$ where $u_0 = K_0 \varphi$ and 
$$
u_1(x) = \int_W G(x,y) \varphi(y) \,dV(y).
$$
Then $-\Delta_g u_1 = \varphi$ in $W$ with $u_1|_{\partial W} = 0$. It follows that $v \in H^1(W)$, $\norm{v}_{H^1(W)} \leq C \norm{\varphi}_{L^2(U)}$, and $\Delta_g v = 0$ in $W$. By elliptic regularity and Sobolev embedding, 
$$
\norm{\nabla^k v}_{L^{\infty}(U)} \leq C_k \norm{v}_{L^2(W)} \leq C_k \norm{\varphi}_{L^2(U)}.
$$
This may be rewritten as 
$$
\left\lvert \int_W \nabla_x^k R(x,y) \varphi(y) \,dV(y) \right\rvert \leq C_k \norm{\varphi}_{L^2(U)}.
$$
Consequently $\norm{\nabla_x^k R(x,\,\cdot\,)}_{L^2(U)} \leq C_k$ uniformly over $x \in U$. Now $R(x,y) = R(y,x)$, so indeed $R$ is smooth in $\overline{U} \times \overline{U}$.
\end{proof}

{\bf $\tau$-dependent single layer potential.} In \cite{KSaU} (see also Proposition \ref{prop:normestimate_uniqueness_2} below) it is shown that for any $\tau \in \mR$ with $\abs{\tau} \geq 1$ and $\tau^2 \notin \text{Spec}(-\Delta_{g_0})$, and given $\delta > 1/2$, one has a bounded linear operator 
\begin{equation*}
G_{\tau}: L^2_{\delta}(T) \to H^2_{-\delta}(T) \cap H^1_{-\delta,0}(T)
\end{equation*}
such that $u = G_{\tau} f$ is the unique solution in $H^1_{-\infty,0}(T)$ of the equation $e^{\tau x_1} (-\Delta_g) e^{-\tau x_1} u = f$ in $T$, for any $f \in L^2_{\delta}(T)$.


If $\tau$ is as above, we define another operator 
\begin{gather*}
K_{\tau}: L^2_{c}(T) \to H^2_{\text{loc}}(T) \cap H^1_{\text{loc},0}(T), \\
K_{\tau} f = e^{-\tau x_1} G_{\tau} (e^{\tau x_1} f).
\end{gather*}
From the properties of $G_{\tau}$, it immediately follows that $K_{\tau}$ is an inverse for the Laplacian: one has 
$$
-\Delta_g K_{\tau} = \id \quad \text{on } L^2_c(T).
$$
The map $K_{\tau}$ is a \emph{Faddeev type Green operator} on $T$. It differs from the standard Green operator $K_0$ by having exponential factors in its integral kernel. Also, $K_{\tau}-K_0$ maps $L^2_c(T)$ into $C^{\infty}(T)$ since $-\Delta_g(K_{\tau}-K_0)f = 0$ for all $f \in L^2_c(T)$. This suggests the following result. We will do the proof with some care to ensure that the boundary $\partial T$ does not pose a problem.

\begin{lemma} \label{claim:layer_smoothkernel}
$K_{\tau} = K_0 + R_{\tau}$ where $R_{\tau}$ is an integral operator with kernel in $C^{\infty}(T \times T)$.
\end{lemma}
\begin{proof}
We prove the result for $\tau \geq 1$, the case $\tau \leq -1$ being similar. The expression for $G_{\tau}$ in \cite[Proposition 4.1]{KSaU} shows that for any $f \in L^2_{\delta}(T)$ with $\delta > 1/2$ we have 
$$
G_{\tau} f(x_1,x') = -\sum_{l=1}^{\infty} [T_{\tau+\sqrt{\lambda_l}} T_{\tau-\sqrt{\lambda_l}} \tilde{f}(\,\cdot\,,l)](x_1) \phi_l(x')
$$
with convergence in $L^2_{-\delta}(T)$, where for $\mu \in \mR \smallsetminus \{0\}$ we define 
$$
T_{\mu} v(t) = \mF_{\eta}^{-1}\left\{ \frac{1}{i\eta - \mu} \hat{v}(\eta) \right\}(t), \quad v \in \mSp(\mR).
$$
This implies that for $f \in L^2_c(T)$ one has 
\begin{align*}
K_{\tau} f(x_1,x') &= -\sum_{l=1}^{\infty} e^{-\tau x_1} [T_{\tau+\sqrt{\lambda_l}} T_{\tau-\sqrt{\lambda_l}} (e^{\tau\,\cdot\,} \tilde{f}(\,\cdot\,,l))](x_1) \phi_l(x'), \\
K_0 f(x_1,x') &= -\sum_{l=1}^{\infty} [T_{\sqrt{\lambda_l}} T_{-\sqrt{\lambda_l}} \tilde{f}(\,\cdot\,,l)](x_1) \phi_l(x')
\end{align*}
with convergence in $L^2_{\text{loc}}(T)$. The second identity follows by solving the equation $-\Delta_g u = f$ in $T$, $f \in L^2(T)$, $u|_{\partial T} = 0$, by taking the Fourier transform in $x_1$ and expanding in terms of Dirichlet eigenfunctions in $x'$. Note that $K_0$ is obtained from $K_{\tau}$ by setting $\tau = 0$ formally.

We compute for $v \in L^2_c(\mR)$ 
\begin{multline*}
- e^{-\tau x_1} [T_{\tau+\sqrt{\lambda_l}} T_{\tau-\sqrt{\lambda_l}} (e^{\tau\,\cdot\,} v)](x_1) + T_{\sqrt{\lambda_l}} T_{-\sqrt{\lambda_l}} v(x_1) \\
 = \int_{-\infty}^{\infty} [e^{-\tau(x_1-y_1)} m_{\tau}(x_1-y_1,\sqrt{\lambda_l}) - m_0(x_1-y_1,\sqrt{\lambda_l})] v(y_1) \,dy_1
\end{multline*}
where for $a > 0$ 
$$
m_{\tau}(t,a) = -\frac{1}{2\pi} \int_{-\infty}^{\infty} e^{it\eta} \frac{1}{(i\eta-(\tau+a))(i\eta-(\tau-a))} \,d\eta.
$$
We claim that 
\begin{equation} \label{residue_claim}
e^{-\tau r} m_{\tau}(r,a) - m_0(r,a) = \left\{ \begin{array}{cl} 0, & \tau < a, \\ -\frac{1}{2 a} e^{-ar}, & \tau > a. \end{array} \right.
\end{equation}
If this holds, then using that $\tau \neq \sqrt{\lambda_l}$ by the condition $\tau^2 \notin \text{Spec}(-\Delta_{g_0})$ we obtain 
\begin{align*}
(K_{\tau} - K_0) f(x_1,x') &= -\sum_{l; \sqrt{\lambda_l} < \tau} \int_{-\infty}^{\infty} \frac{1}{2 \sqrt{\lambda_l}} e^{-\sqrt{\lambda_l}(x_1-y_1)} \tilde{f}(y_1,l) \phi_l(x') \,dy_1 \\
 &= \int_T R_{\tau}(x,y) f(y) \,dV(y)
\end{align*}
where $R_{\tau}(x,y)$ is the integral kernel 
$$
R_{\tau}(x,y) = -\sum_{l; \sqrt{\lambda_l} < \tau} \frac{1}{2 \sqrt{\lambda_l}} e^{-\sqrt{\lambda_l}(x_1-y_1)} \phi_l(x') \phi_l(y').
$$
Since the sum is finite, the kernel is smooth in $T \times T$ as required.

It remains to prove \eqref{residue_claim}. Note that for $r \in \mR$ and $a > 0$,
\begin{align*}
e^{-\tau r} m_{\tau}(r,a) &= -\frac{1}{2\pi} \int_{-\infty}^{\infty} \frac{e^{i(\eta+i\tau)r}}{((\eta+i\tau)i-a)((\eta+i\tau)i+a)} \,d\eta \\
 &= \int_{\gamma_{\tau}} F(z) \,dz
\end{align*}
where the last expression is a contour integral over the curve $\gamma_{\tau}(\eta) = \eta + i\tau$ for $\eta \in (-\infty,\infty)$, and 
$$
F(z) = -\frac{1}{2\pi} \frac{e^{irz}}{(iz-a)(iz+a)}.
$$
It follows that 
$$
e^{-\tau r} m_{\tau}(r,a) - m_0(r,a) = \left( \int_{\gamma_{\tau}} - \int_{\gamma_0} \right) F(z) \,dz.
$$
We may interpret the right hand side as a limit of integrals over closed rectangular contours since $F(\pm R + it)$ where $0 \leq t \leq \tau$ decays as $R \to \infty$. Now $F(z)$ is analytic in $\mC \smallsetminus \{\pm ia\}$ with simple poles at $\pm ia$, so \eqref{residue_claim} follows from the residue theorem.
\end{proof}

Since $R_{\tau} u|_{\partial T} = 0$ for any $u$, the preceding result implies that 
$$
K_{\tau}: H^{-1}_c(T) \to H^1_{\text{loc},0}(T).
$$

\begin{definition}
If $\abs{\tau} \geq 1$ and $\tau^2 \notin \text{Spec}(-\Delta_{g_0})$, we define the $\tau$-dependent single layer potential 
$$
S_{\tau} = K_{\tau} \gamma^*: H^{-1/2}(\Gamma) \to H^1_{\text{loc},0}(T).
$$
\end{definition}

The basic properties of $S_{\tau}$ follow immediately from the previous results.

\begin{lemma} \label{claim:layer_basic}
Let $\abs{\tau} \geq 1$ and $\tau^2 \notin \text{Spec}(-\Delta_{g_0})$. If $f \in H^{-1/2}(\Gamma)$ then $u = S_{\tau} f \in H^1_{\text{loc},0}(T)$ satisfies 
\begin{align*}
-\Delta_g u &= 0 \quad \text{in } M_{\pm}, \\
[u]_{\Gamma} &= 0, \\
[\partial_{\nu} u]_{\Gamma} &= f.
\end{align*}
If $f \in H^s(\Gamma)$ for $s \geq -1/2$ then $u|_{M_-} \in H^{s+3/2}(M_-)$ and $u|_{M_+} = \tilde{u}|_{M_+}$ for some $\tilde{u} \in H^{s+3/2}_{\text{loc}}(T) \cap H^1_{\text{loc},0}(T)$. The trace single layer potential satisfies 
$$
\gamma S_{\tau}: H^s(\Gamma) \to H^{s+1}(\Gamma), \quad s \geq -1/2.
$$
One has 
$$
S_{\tau} f(x) = \int_{\Gamma} K_{\tau}(x,y) f(y) \,dS(y)
$$
where the kernel $K_{\tau}(x,y)$ is smooth off the diagonal in $T^{\text{int}} \times T^{\text{int}}$ and $\Delta_g (K_{\tau}(x,\,\cdot\,)) = 0$ in $T^{\text{int}} \smallsetminus \{x\}$.
\end{lemma}

\section{Boundary integral equation} \label{section_boundaryintegralequation}

We now describe four equivalent problems for characterizing the CGO solutions $u$: a differential equation (DE), integral equation (IE), exterior problem (EP), and boundary integral equation (BE). In this section we assume that $(M,g) \subset \subset (T^{\text{int}},g)$ is a compact manifold with smooth boundary and $T = \mR \times M_0$, $g = e \oplus g_0$, and $(M_0,g_0)$ is any compact $(n-1)$-dimensional manifold with smooth boundary. We use the notations in Section \ref{section_boundarylayerpotentials}.

As a first step, we quote the basic existence and uniqueness result concerning $G_{\tau}$ from \cite[Proposition 4.3 and subsequent remark]{KSaU}.

\begin{prop} \label{prop:normestimate_uniqueness_2}
Let $\delta > 1/2$ and $\lambda \neq 0$, and suppose that $q \in L^{\infty}(T)$ is compactly supported. There exists $\tau_0 \geq 1$ (if $q=0$ then $\tau_0 = 1$) such that whenever 
\begin{equation*}
\abs{\tau} \geq \tau_0 \quad \text{and} \quad \tau^2 \notin \text{Spec}(-\Delta_{g_0}),
\end{equation*}
then for any $f = f_1 + f_2$ where $f_1 \in L^2_{\delta}(T)$, $f_2 \in L^2_{-\delta}(T)$, and 
\begin{equation*}
\text{$\mF_{x_1} f_2(\,\cdot\,,x')$ has support in $\{ \abs{\xi_1} \geq \abs{\lambda} \}$ for a.e.~$x' \in M_0$,}
\end{equation*}
there is a unique solution $w \in H^1_{-\infty,0}(T)$ of the equation 
\begin{equation*}
e^{\tau x_1} (-\Delta + q) e^{-\tau x_1} w = f \quad \text{in } T.
\end{equation*}
Further, $w \in H^1_{-\delta,0}(T) \cap H^2_{-\delta}(T)$, and one has $w = G_{\tau} v$ where 
\begin{equation*}
(\id + qG_{\tau}) v = f, \qquad \norm{(\id+qG_{\tau})^{-1}}_{L^2_{\delta} \to L^2_{\delta}} \leq 2.
\end{equation*}
Finally, $w$ satisfies the estimates 
\begin{equation*}
\norm{w}_{H^s_{-\delta}(T)} \leq C \abs{\tau}^{s-1} \left[ \norm{f_1}_{L^2_{\delta}(T)} + \norm{f_2}_{L^2_{-\delta}(T)} \right], \quad 0 \leq s \leq 2,
\end{equation*}
with $C$ independent of $\tau$ and $f_1$, $f_2$.
\end{prop}

In the following result we extend $q \in L^{\infty}(M)$ by zero into $T$, and let $\tau_0$ be as in Proposition \ref{prop:normestimate_uniqueness_2}. The result considers CGO solutions of the form 
$$
u = u_0 + e^{-\tau x_1} r
$$
where $u_0$ is any harmonic function in $H^2_{\text{loc}}(T)$, and $r \in H^1_{-\infty,0}(T)$. We will only fix the choice of the free solution $u_0$ in the next section.

\begin{prop} \label{prop:equivalent_problems}
Let $q \in L^{\infty}(M)$ be such that $0$ is not a Dirichlet eigenvalue of $-\Delta_g + q$ in $M$, and let $\abs{\tau} \geq \tau_0$ and $\tau^2 \notin \text{Spec}(-\Delta_{g_0})$. Further, let $u_0 \in H^2_{\text{loc}}(T)$ be such that $\Delta_g u_0 = 0$ in $T$. Consider the following problems:
\begin{align*}
\text{(DE)}\ & \left\{ \begin{array}{l}
(-\Delta_g + q) u = 0 \text{ in } T \\[2pt] 
e^{\tau x_1} (u - u_0) \in H^1_{-\infty,0}(T),
\end{array} \right. \\ 
\text{(IE)\,}\ & \left\{ \begin{array}{l} 
u + K_{\tau}(qu) = u_0 \text{ in } T \\[1pt] 
u \in H^2_{\mathrm{loc}}(T),
\end{array} \right. \\ 
\text{(EP)}\ & \left\{ \begin{array}{rl} 
\text{i)} & \Delta_g u = 0 \text{ in } T \smallsetminus M \\[2pt]
\text{ii)} & u = \tilde{u}|_{T \smallsetminus M} \text{ for some } \tilde{u} \in H^2_{\text{loc}}(T) \\[2pt]
\text{iii)} & e^{\tau x_1} (u - u_0)|_{T \smallsetminus M} = \tilde{r}|_{T \smallsetminus M} \text{ for some } \tilde{r} \in H^1_{-\infty,0}(T) \\[2pt]
\text{iv)} & (\partial_{\nu} u)_+ = \Lambda_{g,q}(\gamma_+ u) \text{ on } \Gamma,
\end{array} \right. \\ 
\text{(BE)}\ & \left\{ \begin{array}{l} 
(\id + \gamma S_{\tau} (\Lambda_{g,q} - \Lambda_{g,0})) f = u_0 \text{ on } \Gamma \\[2pt] 
f \in H^{3/2}(\Gamma). 
\end{array} \right.
\end{align*}
Each of these problems has a unique solution. Further, these problems are equivalent in the sense that $u$ solves (DE) iff $u$ solves (IE), if $u$ solves (DE) then $u|_{T \smallsetminus M}$ solves (EP), if $u$ solves (EP) then there is a solution $\tilde{u}$ of (DE) with $\tilde{u}|_{T \smallsetminus M} = u$, if $u$ solves (DE) then $f = u|_{\Gamma}$ solves (BE), and finally if $f$ solves (BE) then there is a solution $u$ of (DE) with $u|_{\Gamma} = f$.
\end{prop}
\begin{proof}
The function $u = u_0 + e^{-\tau x_1} r$ solves $(-\Delta_g + q)u = 0$ in $T$ if and only if  
$$
e^{\tau x_1} (-\Delta_g + q) e^{-\tau x_1} r = - e^{\tau x_1} q u_0 \quad \text{in } T.
$$
The right hand side is in $L^2_c(T)$, so by Proposition \ref{prop:normestimate_uniqueness_2} there is a unique solution $r \in H^1_{-\infty,0}(T)$. This proves that (DE) has a unique solution. It remains to prove that all four problems are equivalent in the sense described above.

(DE) $\implies$ (IE): Assume $u$ solves (DE). Then $u = u_0 + e^{-\tau x_1} r$ where $r \in H^1_{-\infty,0}(T)$, and 
$$
e^{\tau x_1} (-\Delta_g + q) e^{-\tau x_1} r = - e^{\tau x_1} q u_0 \quad \text{in } T.
$$
By Proposition \ref{prop:normestimate_uniqueness_2} we have $r = G_{\tau} v \in H^2_{\text{loc}}(T)$ where $v$ satisfies 
$$
v + q r = -e^{\tau x_1} q u_0.
$$
Since $q$ is compactly supported in $T$ also $v = -q(r+e^{\tau x_1} u_0)$ is compactly supported. Thus we may apply $G_{\tau}$ to both sides of the last identity to obtain 
\begin{equation*}
r + G_{\tau}(qr) = -G_{\tau}(e^{\tau x_1} q u_0).
\end{equation*}
Multiplying by $e^{-\tau x_1}$ and adding $u_0$ to both sides gives (IE).

(IE) $\implies$ (DE): Assume $u$ solves (IE). Then the function $r = e^{\tau x_1}(u-u_0)$ satisfies 
\begin{equation} \label{equiv_r_eq}
r = -G_{\tau}(e^{\tau x_1} q u).
\end{equation}
This shows that $r \in H^1_{-\infty,0}(T)$, and (DE) follows by applying $-\Delta_g$ to both sides of (IE).

(DE) $\implies$ (EP): Let $\tilde{u}$ solve (DE), and define $u = \tilde{u}|_{T \smallsetminus M}$. Clearly properties i), ii) and iii) of (EP) are valid. We need to show iv). Since $\tilde{u}$ solves the equation $(-\Delta_g + q)\tilde{u} = 0$ in $M$, we have 
\begin{equation*}
(\partial_{\nu} u)_+ = \partial_{\nu} \tilde{u}|_{\Gamma} = \Lambda_{g,q} (\tilde{u}|_{\Gamma}) = \Lambda_{g,q} (\gamma_+ u).
\end{equation*}

(EP) $\implies$ (DE): Suppose $u$ solves (EP). Define $v \in H^2(M)$ as the unique solution of the equation $(-\Delta_g + q)v = 0$ in $M$ with $v|_{\Gamma} = \gamma_+ u|_{\Gamma}$, and define 
\begin{equation*}
\tilde{u}(x) = \left\{ \begin{array}{ll} v(x), & x \in M, \\ u(x), & x \in T \smallsetminus M. \end{array} \right.
\end{equation*}
Then $\gamma_- \tilde{u}|_{\Gamma} = \gamma_+ \tilde{u}|_{\Gamma}$ and 
\begin{equation*}
(\partial_{\nu} \tilde{u})_-|_{\Gamma} = \Lambda_{g,q} (\gamma_+ u|_{\Gamma}) = (\partial_{\nu} \tilde{u})_+|_{\Gamma}
\end{equation*}
by (EP) iv). It follows that $\tilde{u} \in H^2_{\text{loc}}(T)$ and $(-\Delta_g + q)\tilde{u} = 0$ in $T$. Further, $e^{\tau x_1} (\tilde{u} - u_0) \in H^1_{-\infty,0}(T)$ by (EP) iii).

(DE) $\implies$ (BE): Let $u$ solve (DE), and let $f = u|_{\Gamma}$. We fix a point $x \in T^{\text{int}} \smallsetminus M$ and let $v(y) = K_{\tau}(x,y)$ where $y \in M$. This is a smooth function in $M$ by Lemma \ref{claim:layer_basic}.

Now Green's theorem implies 
\begin{equation*}
\int_{\Gamma} (u \partial_{\nu} v - v \partial_{\nu} u) \,dS = \int_M (u \Delta_g v - v \Delta_g u) \,dV.
\end{equation*}
By (DE) we have $\Delta_g u = qu$ and $\partial_{\nu} u|_{\Gamma} = \Lambda_{g,q} f$. Using the properties in Lemma \ref{claim:layer_basic} we obtain 
\begin{equation*}
\int_{\Gamma} u \partial_{\nu} v \,dS - S_{\tau} \Lambda_{g,q} f(x) = - K_{\tau} (qu)(x),
\end{equation*}
which is valid for $x \in T^{\text{int}} \smallsetminus M$. The function $v$ is harmonic in $M$, hence $\partial_{\nu} v|_{\Gamma} = \Lambda_{g,0}(v|_{\Gamma})$. The symmetry of $\Lambda_{g,0}$ implies 
\begin{equation*}
\int_{\Gamma} u \partial_{\nu} v \,dS = \int_{\Gamma} u \Lambda_{g,0}(v|_{\Gamma}) \,dS = \int_{\Gamma} \Lambda_{g,0} (u|_{\Gamma}) v \,dS = S_{\tau} \Lambda_{g,0} f(x).
\end{equation*}
We obtain 
\begin{equation} \label{stau_difference_ktau}
S_{\tau}(\Lambda_{g,q} - \Lambda_{g,0}) f = K_{\tau}(qu) \quad \text{in } T^{\text{int}} \smallsetminus M.
\end{equation}
Adding $u$ to both sides, using the fact that $u$ solves (IE), and taking traces on $\Gamma$ gives (BE).

(BE) $\implies$ (EP): Let $f$ solve (BE). We define a function $\tilde{u} \in H^1_{\text{loc}}(T)$ by 
\begin{equation*}
\tilde{u} = u_0 - S_{\tau} (\Lambda_{g,q}-\Lambda_{g,0}) f.
\end{equation*}
This function is harmonic in $T \smallsetminus \Gamma$ by Lemma \ref{claim:layer_basic}, and $\tilde{u}|_{\Gamma} = f$ by using (BE). The jump relation for $S_{\tau}$ implies that on $\Gamma$ 
\begin{equation*}
(\partial_{\nu} \tilde{u})_- - (\partial_{\nu} \tilde{u})_+ = -(\Lambda_{g,q}-\Lambda_{g,0})f.
\end{equation*}
But $(\partial_{\nu} \tilde{u})_- = \Lambda_{g,0}f$, so we have $(\partial_{\nu} \tilde{u})_+ = \Lambda_{g,q} (\gamma_+ \tilde{u})$. Therefore $\tilde{u}|_{T \smallsetminus M}$ satisfies (EP) i) and iv). Also (EP) ii) is valid by mapping properties of $S_{\tau}$.

To prove (EP) iii) it is sufficient to show that for any $h \in H^{1/2}(\Gamma)$, 
\begin{equation} \label{etaux1_Stau_property}
e^{\tau x_1} S_{\tau} h|_{T \smallsetminus M} = w|_{T \smallsetminus M} \quad \text{for some } w \in H^1_{-\infty,0}(T).
\end{equation}
Formally one has $e^{\tau x_1} S_{\tau} h = G_{\tau} e^{\tau x_1} \gamma^* h$ where $G_{\tau}$ maps $L^2_c(T)$ to $H^1_{-\infty,0}(T)$. However, we have not proved that $G_{\tau}$ has good mapping properties on negative order Sobolev spaces. Thus, the proof proceeds differently and involves an extension $w$ of $e^{\tau x_1} S_{\tau} h|_{T \smallsetminus M}$ into $T$ such that $w = G_{\tau} \psi$ for some $\psi \in L^2_c(T)$. This will imply \eqref{etaux1_Stau_property} by the mapping properties of $G_{\tau}$.

Define 
$$
w(x) = \left\{ \begin{array}{ll} e^{\tau x_1}(P_0(\gamma S_{\tau} h) + F), & x \in M, \\ e^{\tau x_1} S_{\tau} h, & x \in T \smallsetminus M \end{array} \right.
$$
where $P_0: H^{3/2}(\Gamma) \to H^2(M)$ is the Poisson operator mapping $h_0$ to the function $v_0$ with $-\Delta_g v_0 = 0$ in $M$ and $v_0|_{\Gamma} = h_0$, and $F \in H^2(M)$ is a function chosen so that $e^{-\tau x_1} w \in H^2_{\text{loc}}(T)$. Clearly we need that $F|_{\Gamma} = 0$, and since 
$$
[\partial_{\nu} (e^{-\tau x_1} w)]_{\Gamma} = \Lambda_{g,0}(\gamma S_{\tau} h) + \partial_{\nu} F|_{\Gamma} - (\partial_{\nu} S_{\tau} h)_+
$$
we also require that $\partial_{\nu} F|_{\Gamma} = (\partial_{\nu} S_{\tau} h)_+ - \Lambda_{g,0}(\gamma S_{\tau} h) \in H^{1/2}(\Gamma)$. We can take $F$ to be any function in $H^2(M)$ with this Cauchy data, and then $e^{-\tau x_1} w$ and also $w$ is in $H^2_{\text{loc}}(T)$.

We now observe that 
$$
e^{\tau x_1} (-\Delta_g) e^{-\tau x_1} w(x) = \left\{ \begin{array}{ll} - e^{\tau x_1} \Delta_g F, & x \in M, \\ 0, & x \in T \smallsetminus M. \end{array} \right.
$$
Since $w \in H^2_{\text{loc}}(T)$ this implies that $e^{\tau x_1} (-\Delta_g) e^{-\tau x_1} w = \psi$ where $\psi \in L^2_c(T)$. Consequently $w = G_{\tau} \psi \in H^1_{-\infty,0}(T)$ and we have proved \eqref{etaux1_Stau_property}.
\end{proof}

Finally, let us verify that the boundary integral equation (BE) in Proposition \ref{prop:equivalent_problems} is indeed Fredholm.

\begin{prop}
The operator 
$$
\gamma S_{\tau} (\Lambda_{g,q} - \Lambda_{g,0}): H^{3/2}(\Gamma) \to H^{3/2}(\Gamma)
$$
is compact.
\end{prop}
\begin{proof}
Let $f \in H^{3/2}(\Gamma)$, and let $u = P_q f$ where $P_q: H^{3/2}(\Gamma) \to H^2(M)$ is the Poisson operator mapping $h_0$ to $v_0$ where $(-\Delta_g + q)v_0 = 0$ in $M$ and $v_0|_{\Gamma} = h_0$. The exact same argument leading to \eqref{stau_difference_ktau} in the proof of Proposition \ref{prop:equivalent_problems} shows that 
$$
S_{\tau}(\Lambda_{g,q} - \Lambda_{g,0}) f = K_{\tau}(q E J u) \quad \text{in } T^{\text{int}} \smallsetminus M
$$
where $E: L^2(M) \to L^2(T)$ is extension by zero and $J: H^2(M) \to L^2(M)$ is the natural inclusion. Taking traces on $\Gamma$, we obtain the factorization 
$$
\gamma S_{\tau} (\Lambda_{g,q} - \Lambda_{g,0}) = \gamma K_{\tau} q E J P_q.
$$
The result follows since on the right hand side $J$ is compact and all other operators are bounded.
\end{proof}

\section{Proofs of the main results} \label{section_proofs}

In Sections \ref{section_boundarylayerpotentials} and \ref{section_boundaryintegralequation} we considered layer potentials and equivalent problems characterizing CGO solutions in the case where $(M,g) \subset \subset (T^{\text{int}},g)$ where $T = \mR \times M_0$, $g = e \oplus g_0$, and $(M_0,g_0)$ can be any compact $(n-1)$-dimensional manifold with boundary. Now we specialize to the case where $(M_0,g_0)$ is simple and prove Theorems \ref{thm_main1} and \ref{thm_main2}.

The first step is to fix the harmonic functions $u_0$ used in Proposition \ref{prop:equivalent_problems}. We first choose a simple manifold $(\tilde{M}_0,g_0)$ such that $(M_0,g_0) \subset \subset (\tilde{M}_0,g_0)$. Below, we will write $(r,\theta)$ for the polar normal coordinates in $(\tilde{M}_0,g_0)$ with center at a given point $p \in \tilde{M}_0 \smallsetminus M_0$ (these exist globally because the manifold is simple), and we write, following \cite[Section 5]{DKSaU}, 
\begin{equation*}
\tilde{a} = \tilde{a}(x_1,r,\theta) = e^{-i\tau r} \abs{g}^{-1/4} e^{i\lambda(x_1+ir)} b(\theta)
\end{equation*}
where $\lambda$ is a fixed nonzero real number and $b \in C^{\infty}(S^{n-2})$ is a fixed function. Note that $\tilde{a} \in C^{\infty}(T)$ since the coordinates $(r,\theta)$ are smooth in $M_0$.

\begin{prop} \label{prop:freesolutions}
Given any $\tau$ with $\abs{\tau} \geq 1$ and $\tau^2 \notin \text{Spec}(-\Delta_{g_0})$, and for any point $p \in \tilde{M}_0 \smallsetminus M_0$, for any real number $\lambda \neq 0$, and for any smooth function $b = b(\theta)$, there is a function $u_0$ with 
$$
\Delta_g u_0 = 0 \ \ \text{in } T, \qquad u_0 \in H^2_{\text{loc}}(T),
$$
of the form 
$$
u_0 = e^{-\tau x_1} \tilde{a} + e^{-\tau x_1} r_0
$$
where $r_0 = G_{\tau} f$ for some explicit function $f$ and $\norm{r_0}_{L^2(M)} = O(\abs{\tau}^{-1})$ as $\abs{\tau} \to \infty$.
\end{prop}
\begin{proof}
If $u_0$ is of the required form, then $\Delta_g u_0 = 0$ is equivalent with 
\begin{equation} \label{u0_r0_equation}
e^{\tau x_1} (-\Delta_g) e^{-\tau x_1} r_0 = f
\end{equation}
where $f = e^{\tau x_1} \Delta_g (e^{-\tau x_1} \tilde{a})$. Writing $\Phi = x_1 + ir$, we compute 
\begin{multline*}
f = -e^{-i\tau r} [e^{\tau \Phi}(-\Delta_g) e^{-\tau \Phi}](\abs{g}^{-1/4} e^{i\lambda(x_1+ir)} b(\theta)) \\
 = -e^{-i\tau r} [-\tau^2 \langle d\Phi,d\Phi \rangle + \tau (2\langle d\Phi,d\,\cdot\, \rangle + \Delta_g \Phi) - \Delta_g](\abs{g}^{-1/4} e^{i\lambda(x_1+ir)} b(\theta)).
\end{multline*}
Here we have extended $\langle \,\cdot\,,\,\cdot\, \rangle$ as a complex bilinear form to complex valued $1$-forms. As in \cite[Section 5]{DKSaU}, we see that $\langle d\Phi,d\Phi \rangle = 0$ and also that $(2\langle d\Phi,d\,\cdot\, \rangle + \Delta_g \Phi)(\abs{g}^{-1/4} e^{i\lambda(x_1+ir)} b(\theta)) = 0$ (this was the reason for the choice of $\tilde{a}$). Consequently 
\begin{align*}
f &= e^{-i\tau r} \Delta_g (e^{i\tau r} \tilde{a}) = e^{-i\tau r} (\partial_1^2 + \Delta_{g_0}) (e^{i\tau r} \tilde{a}) \\
 &= e^{i\lambda x_1} [e^{-i\tau r}(\Delta_{g_0}-\lambda^2)(\abs{g}^{-1/4} e^{-\lambda r} b(\theta))].
\end{align*}
Then for any $\delta > 1/2$ one has $f \in L^2_{-\delta}(T)$, the norm $\norm{f}_{L^2_{-\delta}(T)}$ is independent of $\tau$, and the Fourier transform $\mF_{x_1} f(\,\cdot\,,x')$ is supported in $\{\abs{\xi_1} \geq \abs{\lambda} \}$. By Proposition \ref{prop:normestimate_uniqueness_2} we have a solution $r_0 = G_{\tau} f$ of \eqref{u0_r0_equation}, which gives the required solution $u_0$.
\end{proof}

We can now prove the main theorems.

\begin{proof}[Proof of Theorem \ref{thm_main2}]
We first consider the case, as in the beginning of this section, where $(M,g)$ is an admissible manifold with conformal factor $c = 1$. Suppose that the manifold $(M,g)$, and consequently also $(M_0,g_0)$, and the map $\Lambda_{g,q}$ are known. We wish to determine $q$ from this knowledge.

First note the basic integral identity (see \cite[Lemma 6.1]{DKSaU})
\begin{equation} \label{main2proof_integral_identity}
\int_{\partial M} ((\Lambda_{g,q} - \Lambda_{g,0}) f_1) f_2 \,dS = \int_M q u_1 u_2 \,dV
\end{equation}
which is valid for any $u_j \in H^2(M)$ with $(-\Delta+q)u_1 = 0$ in $M$, $\Delta_g u_2 = 0$ in $M$, and $u_j|_{\partial M} = f_j$. We consider CGO solutions in $T$ of the form 
\begin{gather*}
u_1 = u_{0,1} + e^{-\tau x_1} r_1, \\
u_2 = u_{0,2}
\end{gather*}
where $u_{0,j}$ are harmonic functions provided by Proposition \ref{prop:freesolutions} having the form 
\begin{gather*}
u_{0,1} = e^{-\tau (x_1+ir)} \abs{g}^{-1/4} e^{i\lambda(x_1+ir)} b(\theta) + e^{-\tau x_1} G_{\tau} \psi_1, \\
u_{0,2} = e^{\tau (x_1+ir)} \abs{g}^{-1/4} e^{i\lambda(x_1+ir)} + e^{\tau x_1} G_{-\tau} \psi_2.
\end{gather*}
Here $\tau \geq \tau_0$ and $\tau^2 \notin \text{Spec}(-\Delta_{g_0})$, $(r,\theta)$ are polar normal coordinates in $(\tilde{M}_0,g_0)$ with center at $p \in \tilde{M}_0 \smallsetminus M_0$, $\lambda \neq 0$, $b$ is a smooth function in $S^{n-2}$, and $\psi_j$ are explicit functions with $\norm{G_{\pm \tau} \psi_j}_{L^2(M)} = O(\abs{\tau}^{-1})$.

The point is that $u_{0,j}$ are explicit functions which can be constructed from the knowledge of $(M,g)$, and also $f_2 = u_{0,2}|_{\Gamma}$ is known. By Proposition \ref{prop:equivalent_problems} there is a unique CGO solution $u_1$ of the above form, and the boundary value $f_1 = u_1|_{\Gamma}$ is the unique solution in $H^{3/2}(\Gamma)$ of the boundary integral equation 
$$
(\id + \gamma S_{\tau} (\Lambda_{g,q} - \Lambda_{g,0})) f_1 = u_{0,1} \quad \text{on } \Gamma.
$$
Since the operator on the left and the function on the right are known from our data, we can construct $f_1$ as the unique solution of this Fredholm integral equation. Then the left hand side of \eqref{main2proof_integral_identity} is known, and consequently we can determine from our data the integrals 
\begin{equation} \label{thm_main2_proof_orthogonality}
\int_M q u_1 u_2 \,dV
\end{equation}
for any $u_1$ and $u_2$ as above.

Since $u_1$ solves $(-\Delta_g + q)u_1 = 0$ with $r_1 \in H^1_{-\infty,0}(T)$, Proposition \ref{prop:normestimate_uniqueness_2} shows that 
$$
r_1 = - G_{\tau} (\id + q G_{\tau})^{-1} (e^{\tau x_1} q u_{0,1})
$$
and using the form of $u_{0,1}$ and norm estimates for $G_{\tau}$ gives 
$$
\norm{r_1}_{L^2(M)} = O(\abs{\tau}^{-1}).
$$
Thus, taking the limit as $\tau \to \infty$ in \eqref{thm_main2_proof_orthogonality}, we have recovered from our boundary data the quantities 
$$
\int_M q \abs{g}^{-1/2} e^{2i\lambda(x_1+ir)} b(\theta) \,dV.
$$

At this point it is convenient to extend $q$ into $T$ as a function in $C^{\infty}_c(T^{\text{int}})$. This may be done by recovering the Taylor series of $q$ on $\partial M$ via boundary determination \cite[Section 8]{DKSaU} (this procedure is constructive), and by extending $q$ to a function in $C^{\infty}_c(T^{\text{int}})$ so that $q|_{T \smallsetminus M}$ is known. Using that $dV = \abs{g}^{1/2} \,dx_1 \,dr \,d\theta$, the last integral becomes 
$$
\int_0^{\infty} \int_{S^{n-2}} e^{-2\lambda r} \left[ \int_{-\infty}^{\infty} e^{2i\lambda x_1} q(x_1,r,\theta) \,dx_1 \right] b(\theta) \,dr \,d\theta.
$$
Denoting the quantity in brackets by $f_{\lambda}(r,\theta)$ and by varying the smooth function $b$, we determine the integrals 
$$
\int_0^{\infty} e^{-2\lambda r} f_{\lambda}(r,\theta) \,dr \quad \text{for all } \theta \in S^{n-2}.
$$

These integrals are known for any nonzero real number $\lambda$ and for any point $p \in \tilde{M}_0 \smallsetminus M_0$ which is the center of the polar normal coordinates $(r,\theta)$ in $\tilde{M}_0$. Noting that $r \mapsto (r,\theta)$ is the unit speed geodesic in $(\tilde{M}_0,g_0)$ starting at $p$ in direction $\theta$, and letting $p$ approach $\partial M_0$, we can recover the integrals 
\begin{equation} \label{thm_main2_proof_attenuatedraytransform}
\int_0^T e^{-2\lambda r} f_{\lambda}(\gamma(r)) \,dr
\end{equation}
for any geodesic $\gamma: [0,T] \to M_0$ where $\gamma(0), \gamma(T) \in \partial M_0$ and $\gamma(t)$ for $0 < t < T$ lies in $M_0^{\text{int}}$. This is the attenuated geodesic ray transform of $f_{\lambda}$ in $(M_0,g_0)$, with constant attenuation $-2\lambda$. Now, assuming $\dim(M) = 3$ so $(M_0,g_0)$ is $2$-dimensional, we invoke the invertibility result for the attenuated ray transform \cite{SaU} which allows to recover the function $f_{\lambda}$ in $M_0$ from the integrals \eqref{thm_main2_proof_attenuatedraytransform} for any $\lambda$. Thus, we have determined the integrals 
$$
\int_{-\infty}^{\infty} e^{2i\lambda x_1} q(x_1,x') \,dx_1
$$
for any $\lambda \neq 0$ and for any $x' \in M_0$. This determines $q$ in $M$ by inverting the one-dimensional Fourier transform.

We have proved the theorem in the case where $(M,g)$ is an admissible manifold and with conformal factor $c = 1$. For general conformal factors, suppose that $(M,g)$ is admissible and $g = c\tilde{g}$ where $\tilde{g} = e \oplus g_0$. Define also $\tilde{q} = c(q-q_c)$ where $q_c = c^{\frac{n-2}{4}} \Delta_{c \tilde{g}}(c^{-\frac{n-2}{4}})$. The identity 
$$
c^{\frac{n+2}{4}} (-\Delta_{c\tilde{g}} + q) (c^{-\frac{n-2}{4}} u) = (-\Delta_{\tilde{g}} + \tilde{q})u
$$
implies that, since $\nu_{\tilde{g}} = c^{1/2} \nu_g$, 
\begin{equation} \label{thm_main2_proof_reduction_identity}
\Lambda_{\tilde{g},\tilde{q}} f = c^{\frac{n}{4}} \Lambda_{g,q}(c^{-\frac{n-2}{4}} f) + \frac{n-2}{4} c^{-1} (\partial_{\nu_{\tilde{g}}} c) f.
\end{equation}
Thus, from the knowledge of $\Lambda_{g,q}$ and $(M,g)$ we can determine $\Lambda_{\tilde{g},\tilde{q}}$. The proof above then shows that one can reconstruct $\tilde{q}$, from which $q$ is easily determined.
\end{proof}

\begin{proof}[Proof of Theorem \ref{thm_main1}]
Let $(M,\tilde{g})$ be admissible and known and suppose that $\Lambda_{c\tilde{g}} = \Lambda_{c\tilde{g},0}$ is known. By boundary determination \cite{DKSaU} we can determine $c|_{\partial M}$ and $\partial_{\nu_{\tilde{g}}} c|_{\partial M}$. The identity \eqref{thm_main2_proof_reduction_identity} shows that 
$$
\Lambda_{\tilde{g},\tilde{q}} f = c^{\frac{n}{4}} \Lambda_{c\tilde{g},0}(c^{-\frac{n-2}{4}} f) + \frac{n-2}{4} c^{-1} (\partial_{\nu_{\tilde{g}}} c) f
$$
with $\tilde{q} = -c^{\frac{n+2}{4}} \Delta_{c\tilde{g}}(c^{-\frac{n-2}{4}})$. This shows that $\Lambda_{c\tilde{g}}$ determines $\Lambda_{\tilde{g},\tilde{q}}$, and Theorem \ref{thm_main2} implies that we can recover $\tilde{q}$.

We write $w = \log\,c^{-\frac{n-2}{4}}$ and compute 
\begin{align*}
\Delta_{\tilde{g}} w &= \sum_{j,k=1}^n \abs{\tilde{g}}^{-1/2} \partial_j(\abs{\tilde{g}}^{1/2} \tilde{g}^{jk} c^{\frac{n-2}{4}} \partial_k (c^{-\frac{n-2}{4}})) \\
 &= \sum_{j,k=1}^n c^{n/2} \abs{c\tilde{g}}^{-1/2} \partial_j(c^{-\frac{n-2}{4}}\abs{c\tilde{g}}^{1/2} (c\tilde{g})^{jk} \partial_k (c^{-\frac{n-2}{4}})) \\
 &= c^{\frac{n+2}{4}} \Delta_{c\tilde{g}}(c^{-\frac{n-2}{4}}) + \sum_{j,k=1}^n c^{\frac{n-2}{2}} \tilde{g}^{jk} \partial_j(c^{-\frac{n-2}{4}}) \partial_k(c^{-\frac{n-2}{4}}).
\end{align*}
This implies that 
\begin{align*}
-\Delta_{\tilde{g}} w + \langle dw, dw \rangle_{\tilde{g}} &= \tilde{q} \qquad \text{in } M, \\
w|_{\partial M} &= \log\,c^{-\frac{n-2}{4}}|_{\partial M}.
\end{align*}
This nonlinear Dirichlet problem has a unique solution by the maximum principle \cite{GilbargTrudinger} and we have already recovered the right hand side $\tilde{q}$ and the boundary value $\log\,c^{-\frac{n-2}{4}}|_{\partial M}$, so we may construct $w$ in $M$ by solving the problem. This determines $c$ in $M$.
\end{proof}


\providecommand{\bysame}{\leavevmode\hbox to3em{\hrulefill}\thinspace}
\providecommand{\href}[2]{#2}

\end{document}